\documentclass[a4paper, reqno, 12pt]{amsart}
\usepackage[english]{babel}
\usepackage[T1]{fontenc}
\usepackage{mathptmx,amsmath,dsfont}
\usepackage{hyperref}
\usepackage{graphicx}
\usepackage{amssymb,verbatim,cases,bbold}
\usepackage{fullpage}
\usepackage[svgnames,hyperref]{xcolor}
\newtheorem{theorem}{Theorem}[section]
\newtheorem{corollary}[theorem]{Corollary}
\newtheorem{lemma}[theorem]{Lemma}
\newtheorem{definition}[theorem]{Definition}
\newtheorem{proposition}[theorem]{Proposition}
\newtheorem{question}[theorem]{Question}
\newtheorem{remark}[theorem]{Remark}

\numberwithin{equation}{section}

\newcommand{\R}{\mathbb{R}}

\newcommand{\Z}{\mathbb{Z}}

\newcommand{\capa}{\text{Cap}}
\newcommand{\F}{\mathcal{F}}

\newcommand{\C}{\mathbb{C}}

\newcommand{\ddc}{dd^c}
\newcommand{\psh}{\text{PSH}(\Omega)}

\newcommand{\pshn}{\text{PSH}^-(\Omega)}

\newcommand{\om}{\omega}
\newcommand{\Om}{\Omega}
\newcommand{\NP}[1]{\textnormal{NP}({dd^c#1})^n}
\newcommand{\NNP}{\mathcal{N}_\textnormal{NP}}

\begin{document}
	\author{Thai Duong Do\textit{$^{1}$}, Hoang-Son Do\textit{$^{2}$}, Van Tu Le\textit{$^{3}$}, Ngoc Thanh Cong Pham\textit{$^{4}$}}
	\address{\textit{$^{1}$} Department of Mathematics, National University of Singapore - 10, Lower Kent Ridge Road - Singapore 119076\footnote[5]{On leave from Institute of Mathematics, Vietnam Academy of Science and Technology}}
	\email{duongdothai.vn@gmail.com}
	\address{\textit{$^{2}$} Institute of Mathematics, Vietnam Academy of Science and Technology, 18, Hoang Quoc Viet, Hanoi, Viet Nam}
	\email{hoangson.do.vn@gmail.com, dhson@math.ac.vn}
	\address{\textit{$^{3}$} Faculty of Mathematics and Informatics, Hanoi University of Science and Technology, No. 1 Dai Co Viet,
		 Hai Ba Trung, Ha Noi, Vietnam}
	\email{tu.levan@hust.edu.vn}
	\address{\textit{$^{4}$} Institute of Mathematics, Vietnam Academy of Science and Technology, 18, Hoang Quoc Viet, Hanoi, Viet Nam}
	\email{cong.pnt.math@gmail.com, phamngocthanhcong1997@gmail.com}
	\title{A Dirichlet type problem for non-pluripolar complex Monge-Amp\`ere equations}
	
	
	.
	\date{\today}
	\begin{abstract} In this paper, we study  a Dirichlet type problem for the non-pluripolar complex Monge - Amp\`ere equation  with
		 prescribed singularity on a bounded domain of $\C^n$.  We provide a local version for an existence and uniqueness theorem proved by  Darvas, Di Nezza and Lu in \cite{DDL21}.
	Our work also extends a result of \AA hag, Cegrell, Czy\.z and Pham in \cite{ACCP}.
	\end{abstract}

	\maketitle
	\tableofcontents
	
	\section{Introduction}
	Let $\Omega$ be a bounded domain in $\C^n$. For each  smooth plurisubharmonic function $u$ on $\Om$, the complex Monge-Amp\`ere operator of $u$ is defined by 
	$$(dd^c u)^n =C_n \det(Hu) dV,$$ 
	where  $Hu$ is the complex Hessian of  $u$, $dV$ is the standard volume form and  $C_n>0$ is a constant depending only on $n$.

	Bedford and Taylor \cite{BT76, BT82}  have extended the concept of the complex Monge-Amp\`ere  operator for bounded plurisubharmonic function, whereby $(dd^c u)^n$ is a Radon measure satisfying the following property: If $u_j$ is a sequence of smooth plurisubharmonic functions decreasing to $u$ then $(dd^c u_j)^n$ converges weakly to $(dd^c u)^n$.  The set $\mathcal{D}(\Om)$ of plurisubharmonic functions whose Monge-Amp\`ere operator can be defined as above is called the domain of definition of  Monge-Amp\`ere operator. The characteristics of  the domain of definition of  Monge-Amp\`ere operator  were  studied by Cegrell \cite{Ceg04} and Blocki \cite{Blo06}.
	When $\Om$ is a hyperconvex domain, the class $\pshn\cap\mathcal{D} (\Om)$ is also denoted by
	$\mathcal{E}(\Om)$.
	
	By the comparison principle \cite{BT82}, a bounded plurisubharmonic function is uniquely determined by its Monge-Amp\`ere operator
	and its boundary behavior. Here, we say that $u$ and $v$ have the same boundary behavior if $\lim_{z\to\partial\Om}(u-v)(z)=0$. 
	In particular, the Dirichlet problem
	\begin{equation}\label{Dirichletbounded}
	\begin{cases}
	u\in\psh\cap L^{\infty}(\Om),\\
	(dd^cu)^n=\mu,\\
	\lim_{z\to z_0}u(z)=\varphi (z_0), z_0\in\partial\Om,
	\end{cases}
	\end{equation}
	has at most one solution for every Radon measure $\mu$ on $\Om$ and for every bounded function $\varphi: \partial\Om\rightarrow\R$.
	When $\Om$ is strongly pseudoconvex and $\mu=f dV$ ($f\in L^p(\Om), p>1$), the problem
	\eqref{Dirichletbounded} has a unique continuous (resp., H\"older continuous) solution, provided that $\varphi$ is
	continuous (resp., H\"older continuous) \cite{Kol98, GKZ08, BKPZ16}.  
	
	A generalized Dirichlet problem for Monge-Amp\`ere equation in the class $\mathcal{D}(\Om)$ has been studied
	in \cite{ACCP} (see also \cite{Ceg98, Ceg04, Ahag07}). Assume that $\Omega$ is a bounded hyperconvex domain
	and $H\in\mathcal{E}(\Om)$ is a maximal plurisubharmonic function. By \cite{ACCP}, the Dirichlet problem
	\begin{equation}\label{DirichletCegrell}
		\begin{cases}
			u\in\mathcal{N}(H),\\
			(dd^cu)^n=\mu,
		\end{cases}
	\end{equation}
	has a solution, provided that there exists $\underline{u}\in\mathcal{N}(\Om)$ such that
	$(dd^c \underline{u})^n\geq\mu$. Here, one defines by $\mathcal{N}(\Om)$ the set of functions $v\in\mathcal{E}(\Om)$ with smallest maximal plurisubharmonic majorant identically zero, and
$$\mathcal{N}(H):=\left\{w\in PSH^-(\Om): v+H\leq w\leq H\,\mbox{ for some }\, v\in\mathcal{N}(\Om) \right\}.$$
 It is shown in an example that the problem \eqref{DirichletCegrell} may have many solutions  \cite[page 591]{Zer97}.
However, if $\mu$ vanishes on pluripolar set then the solution to  \eqref{DirichletCegrell} is unique \cite[Theorem 3.7]{ACCP}. 
 It is easy to see that $\limsup_{\Om\setminus N\ni z\to \partial\Om}(u-H)(z)=0$ for every $u\in\mathcal{N}(H)$, 
 where $N=\{H=-\infty\}$, but the converse is not true: if the  condition ``$u\in\mathcal{N}(H)$'' is replaced by
 {\it ``$\limsup_{\Om\setminus N\ni z\to \partial\Om}(u-H)(z)=0$''} then the uniqueness of solution is lost even in the case where $\mu=0$ and $H=0$
 (see \cite{BPT09, BST21}).

	In \cite{BT87}, Bedford and Taylor have studied the plurifine topology, which was first introduced by Fuglede in \cite{Fu86} as the weakest topology in which all plurisubharmonic functions are continuous, and defined the non-pluripolar complex Monge-Amp\`ere measure (also known as ``the non-pluripolar part of Monge-Amp\`ere operator'') {\it for every plurisubharmonic function}. If  $u$ is a negative plurisubharmonic function then  its non-pluripolar Monge-Amp\`ere measure $\NP{u}$  is defined as the limit of the sequence of measures $\mathbb{1}_{\{u>-M\}}(dd^c \max\{u, -M\})^n$ as $M\rightarrow\infty$.  This measure is a Borel measure which puts no mass on pluripolar subsets and may have a locally unbounded mass. If $u$ is a negative plurisubharmonic function belonging to the domain of the complex Monge-Amp\`ere operator then  $\NP{u}= \mathbb{1}_{\{u>-\infty\}}(\ddc u)^n$ (see \cite{BGZ09}). 
	
	The idea behind the definition of the non-pluripolar complex Monge-Amp\`ere measure in the local setting  has been adapted to the case of K\"ahler manifold  \cite{GZ07, BEGZ10}. 
	Consider a complex compact K\"ahler manifold $(X,\omega)$ and let $\theta$ be a closed smooth real $(1,1)$-form  on $X$ such that its cohomology class $\alpha$ is big.
	In \cite{BEGZ10}, Boucksom, Eyssidieux, Guedj and Zeriahi have defined the non-pluripolar complex Monge-Amp\`ere measure $(\theta+dd^cu)^n$ for every $\theta$-plurisubharmonic function $u$. The value of $\int_X\theta_u^n$ is bounded by the volume
	of  $\alpha$ and depends on the singularity of $u$.
	In \cite{DDL18, DDL21}, Darvas, Di Nezza and Lu have studied the complex Monge-Amp\`ere equation with prescribed singularity type:
	\begin{equation}\label{MA}
	\begin{cases}
	(\theta+dd^cu)^n=f\omega^n,\\
	[u]=[\phi],
	\end{cases}
	\end{equation}
	where $\phi$ is a given $\theta$-plurisubharmonic function  and $f\geq0$ is a $L^p$ function $(p>1)$
	satisfying $\int_Xf\om^n=\int_X\theta_{\phi}^n>0$. The notation $[u]=[\phi]$ means that $u$ and $v$ have the
	same singularity type, i.e., $u=v+O(1)$. Darvas, Di Nezza and Lu have introduced the notion of the model potential, the model-type singularity and shown that this equation is well-posed only for potentials $\phi$ with model type singularities, i.e., $[\phi]=[P_{\theta}[\phi]]$,
	where 
	$$P_{\theta}[\phi]=(\sup\{\psi\in PSH(X, \theta): \psi\leq 0, \psi\leq \phi+O(1) \})^*.$$
	They have also emphasized that requiring $\phi$ to be a model potential is not only sufficient, but also a necessary condition for the solvability of (\ref{MA}) for every choice of $f$. Furthermore, they have shown the existence and uniqueness (up to a constant) of solution to the following problem, which is a general form of \eqref{MA} (see \cite[Theorem 4.7]{DDL21}):
	\begin{equation}
	\begin{cases}
	(\theta+dd^cu)^n=\mu,\\
	P_{\theta}[u]=\phi,
	\end{cases}
	\end{equation}
	where $\phi=P_{\theta}[\phi]$ is a model $\theta$-plurisubharmonic function and $\mu$ is a non-pluripolar positive Radon measure on $X$
	satisfying $\int_X\theta_{\phi}^n=\int_Xd\mu>0$. Here, we say that a  measure $\mu$ is  non-pluripolar if it vanishes on every pluripolar
	set. Roughly speaking, the result of Darvas, Di Nezza and Lu tells us that
	 a $\theta$-plurisubharmonic function $u$
	with $\int_X \theta_u^n>0$ is completely determined through $\theta_u^n$, $\sup_X u$ and $P_{\theta}[u]$. The condition on $P_{\theta}[u]$ can be regarded as a mild condition on the singularity of $u$: if $u$ and $v$ have  the same singularity type then 
	$P_{\theta}[u]=P_{\theta}[v]$. We have the following very natural question:
	\begin{question}
		Is every plurisubharmonic function $u\in\pshn$ completely determined through its non-pluripolar Monge-Amp\`ere measure and some mild conditions	on its singularity and boundary behavior? 
	\end{question}
	Inspired of \cite{DDL21},
	we say that a function $u\in\pshn$ is  model if $u=P[u],$ where
	\begin{multline*}
	P[u]=\big(\sup\{v\in\pshn:\ v\leq u+ O(1) \text{ on }\Omega,\ \liminf\limits_{\Omega\setminus N\ni z\rightarrow\xi_0}(u(z) -v(z))\geq0\ \forall\xi_0\in\partial\Omega \}\big)^*,
	\end{multline*}
	and $N=\{u=-\infty\}$. 
	A negative plurisubharmonic function $\phi$ is model iff $\NP{\phi}=0$. Moreover, if $u$ is a negative plurisubharmonic function then the smallest model plurisubharmonic majorant of $u$ is $P[u]$. We refer the reader to Theorem \ref{the: model} below
	for more details.
	
	In this paper, we study the existence and uniqueness of solution to the following Dirichlet type problem for the non-pluripolar
	Monge-Amp\`ere equation 
	\begin{equation}\label{NPMA}
	\begin{cases}
	\NP{u}=\mu,\\
	P[u]=\phi,
	\end{cases}
	\end{equation}
	where $\mu$ is a non-pluripolar positive Borel measure on $\Om$ and $\phi$ is a model plurisubharmonic function on $\Om$.
	
	Denote by $\mathcal{N}_{NP}(\Om)$ (or $\mathcal{N}_{NP}$ for short) the set of negative plurisubharmonic functions $u$ on $\Om$ with smallest model plurisubharmonic majorant
	identically zero (i.e., $P[u]=0$). The main result of this paper is as follows:
	
	\begin{theorem}\label{main}
		Assume that there exists $v\in\pshn$ such that $\NP{v}\geq\mu$ and $P[v]=\phi$. Denote 
	$$S=\{w\in\pshn: w\leq  \phi, \NP{w}\geq\mu\}.$$
	Then $u_S:=(\sup\{w: w\in S\})^*$ is a solution of the problem \eqref{NPMA}. Moreover, if there exists $\psi\in\NNP$ such that 
	$\NP{\psi}\geq\mu$ then $u_S$ is the unique solution of \eqref{NPMA}.
	\end{theorem}
	
	\vspace{0.2cm}
	We stress that in the above theorem, $\phi$ does not necessarily to belong in the class $\mathcal{D}(\Om)$ and $\Omega$ 
	does not need to be hyperconvex. 
	In the case where $\Om$ is hyperconvex, Cegrell has shown that if $\mu$ is a non-pluripolar positive Radon measure on $\Om$ with
	$\mu(\Om)<\infty$ then there exists a unique function $u\in\mathcal{F}$ satisfying $(dd^cu)^n=\mu$ (see
	\cite[Lemma 5.14]{Ceg04}). Here, the class $\mathcal{F}$
	is defined as in \cite[Definition 4.6]{Ceg04}. Actually, $\mathcal{F}(\Om)$ is the set
	of all the functions in $\mathcal{D}(\Om)$ with smallest maximal plurisubharmonic majorant
	identically zero and with finite total Monge-Amp\`ere mass (see, for example, \cite[page 17]{DD21}). By Remark \ref{rmkF} below,
	if $u\in\mathcal{F}$ and $(dd^cu)^n$ vanishes on pluripolar sets then $u\in\NNP$.
	
	Using Theorem \ref{main} and \cite[Lemma 5.14]{Ceg04}, we obtain immediately the following result which can be seen as a local version
	of \cite[Theorem 4.7]{DDL21}:
	\begin{corollary}
		Assume that $\Omega$ is hyperconvex and $\mu$ is a non-pluripolar positive Radon measure on $\Om$ satisfying $\mu (\Om)<\infty$.
		Then, 	there exists a unique plurisubharmonic function $u$ satisfying \eqref{NPMA}. Moreover, $\phi+v\leq u\leq \phi$ for some
		$v\in\mathcal{F}^a(\Om)$.
	\end{corollary}
	 For every $H\in\pshn$, we denote
	$$\mathcal{N}_{NP}(H)=\{w\in\pshn: \mbox{there exists } v\in\mathcal{N}_{NP} \mbox{ such that } v+H\leq w\leq H \},$$
and
		$$\mathcal{N}^a(H)=\{w\in\pshn: \mbox{there exists } v\in\mathcal{N}^a \mbox{ such that } v+H\leq w\leq H \},$$
where	$\mathcal{N}^a$ is the set of functions $v\in\mathcal{D}(\Om)$ with smallest maximal plurisubharmonic majorant identically zero
	and with $(dd^c v)^n$ vanishes on pluripolar sets. It is easy to check that $\mathcal{N}^a\subset\NNP$. 
	The following result, which has been proven first by  \AA hag-Cegrell-Czy\.z-Pham, can be considered as a corollary of Theorem \ref{main}:
	\begin{corollary}\label{cormain}\cite[Theorem 3.7]{ACCP}
		Assume that $\mu$ is a non-negative measure defined on $\Om$ by $\mu=(dd^c\varphi)^n$ for some $\varphi\in\mathcal{N}^a$.
		Then, for every $H\in\mathcal{D}(\Om)$ with $(dd^c H)^n\leq \mu$, there exists a unique function $u\in\mathcal{N}^a(H)$ such that
		$(dd^cu)^n=\mu$ on $\Om$.
	\end{corollary}
	\vspace{0.2cm}
	The paper is organized as follows. In Section \ref{secpre}, we recall auxiliary facts about the plurifine topology and the non-pluripolar Monge-Amp\`ere measure. In Sections \ref{secsta} and \ref{secenvelope}, we introduce some important tools for the
	proof of the existence of solution to \eqref{NPMA}. In Section \ref{sec: comparison}, we prove two Xing-type comparison principles
	and some related results. Theorem \ref{main} and Corollary \ref{cormain} are proved in Section \ref{secproof}.
	\section{Preliminaries}\label{secpre}
	In this section, we recall some basic concepts and properties about the plurifine topology and the non-pluripolar Monge-Amp\`ere measure.
	The reader can find more details in \cite{BT87}.
	\subsection{The plurifine topology}
		The plurifine topology on an open set $\Omega$ in $\C^n$ is the smallest topology on $\Omega$ for which all the plurisubharmonic functions are continuous.
	A basis $\mathcal{B}$ of the plurifine topology on $\Omega$ consists of the sets of the following form:
	\[U\cap\{u>0\}\]
	where $U$ is an open subset in $\Omega,\ u\in\text{PSH}(U)$.
	
The plurifine topology
has the following quasi-Lindel\"of property:
	\begin{theorem}\cite[Theorem 2.7]{BT87}\label{Lindelof}
		An arbitrary union of plurifine open subsets differs from a countable subunion by at most a pluripolar set.
	\end{theorem}
	By the quasi-Lindel\"of property, one get the following lemma:
	\begin{lemma}\label{open subsets decreasing to plurifine open set}
		Let $\mathcal{O}$ be a plurifine open subset of $\Omega$. Then there exists a decreasing sequence $\{V_l\}_l$ of open subsets of $\Omega$
		 such that $V_l$ contains $\mathcal{O}$ for every $l$ and $\cap_{l=1}^\infty V_l\setminus\mathcal{O}$ is a pluripolar set.
	\end{lemma}
	\begin{proof}
		Since we can write $\mathcal{O} = \cup \{ \mathcal{O}_i \in \mathcal{B}, i \in I\}$, it follows from Theorem~\ref{Lindelof} 
		that there exist a sequence $\{\mathcal{O}_j\}_j\subset\mathcal{B}$ and a pluripolar set $N$ such that 
		\begin{equation}\label{O=cup Oj cup N}
		\mathcal{O}=\cup_{j=1}^\infty\mathcal{O}_j\cup N.
		\end{equation}
		By the definition of $\mathcal{B},$  for each $j$, there exist  an open subset $U_j$ of $\Omega$ and  a plurisubharmonic function $u_{j}\in\text{PSH}(U_j)$ such that
		$$\mathcal{O}_j=\{z\in U_j:\ u_{j}(z)>0\}.$$

		Since $u_j$ is quasi-continuous on $U_j$, for every $l\in\Z^+$, there exists an open subset $W_{j,l}$ of $U_j$ such that
		 $\capa(W_{j,l}, U_j) < 2^{-1-l-j}$ and $u_j \in C(U_j \setminus W_{j,l})$. By Tietze's theorem, we can find a continuous extension
		 $f_{j,l}$ of $u_j$  on $U_j$. Set
			\begin{equation}\label{Vjl deacreasing}
			V_{j,l} = \bigcup\limits_{s = l}^{\infty} W_{j,s}.
			\end{equation}
			Then, the sequence $\{V_{j,l}\}_l$ is decreasing and 
			\begin{equation}\label{capa Vjl}
			\capa(V_{j,l},U_j) \leq \sum\limits_{s = l}^{\infty} \capa(W_{j,s}, U_j) \leq \sum\limits_{s = l}^\infty \dfrac{1}{2^{1+s+j}} = 2^{-j-l} \sum\limits_{s = 1}^\infty \dfrac{1}{2^s} = 2^{-j-l}.
			\end{equation}

		Observe that
		\begin{align*}
		\mathcal{O}_j\cup V_{j,l}=&V_{j,l}\cup\{z\in U_j:\ u_{j}(z)>0\}\\
		=&V_{j,l}\cup \{z\in U_j\setminus V_{j,l}:\ u_{j}(z)>0\}\\
		=&V_{j,l}\cup\{z\in U_j\setminus V_{j,l}:\ f_{j,l}(z)>0 \}\\
		=&V_{j,l}\cup \{z\in U_j:\ f_{j,l}(z)>0 \},
		\end{align*}
		which implies that	$\mathcal{O}_j\cup V_{j,l}$  is open.
		
		Let $u\in\pshn$ such that 
		\begin{equation}\label{eq2 lem open}
		N\subset \{u=-\infty\}.
		\end{equation}
		 For each $l\in\Z^+$, we denote $N_l=\{u<-l\}$. We have $N_l$ is open and
		\begin{equation}\label{eq3 lem open}
		\lim_{l\to\infty}	\capa(N_l, \Om)=0.
		\end{equation}
		 Now, for every $l\in\Z^+$, we define
		 $$V_l=N_l\cup_{j=1}^\infty (\mathcal{O}_j\cup V_{j,l}).$$
		Then $\{V_l\}_l$ is a decreasing sequence of open sets. By (\ref{O=cup Oj cup N}) and \eqref{eq2 lem open}, 
		 $\cap_{l=1}^{\infty}V_l$ contains
		 $\mathcal{O}$. Moreover, by (\ref{capa Vjl}), for every $l_0\in\Z^+$,
		 we have
		\begin{align*}
		\capa(\cap_{l=1}^\infty V_l\setminus\mathcal{O},\Omega)
		\leq \capa( V_{l_0}\setminus\mathcal{O},\Omega)
	&\leq \capa(N_{l_0}\cup_{j=1}^\infty  V_{j,l_0},\Omega)\\
 &\leq \capa(N_{l_0}, \Om)+\sum_{j=1}^\infty \capa(V_{j,l_0},\Omega) \\
	&	\leq	\capa(N_{l_0}, \Om)+ \sum_{j=1}^\infty \capa(V_{j,l_0},U_j)\\
		&	\leq	\capa(N_{l_0}, \Om)+\sum_{j=1}^\infty2^{-(l_0+j)}\\
		&=	\capa(N_{l_0}, \Om)+ 2^{-l_0}.
		\end{align*}
		Letting $l_0\rightarrow\infty$ and using  \eqref{eq3 lem open}, we obtain
	$$	\capa(\cap_{l=1}^\infty V_l\setminus\mathcal{O},\Omega)=0.$$
	Hence $\cap_{l=1}^\infty V_l\setminus\mathcal{O}$ is a pluripolar set.
	
		The proof is completed.
	\end{proof}
	\subsection{The non-pluripolar complex Monge-Amp\`ere measure}
	We recall the definition of the non-pluripolar complex Monge-Amp\`ere measures.
	\begin{definition}\label{def NPMA}\cite{BT87}
		If $u\in\psh$ then the non-pluripolar complex Monge-Amp\`ere measure of $u$ is the measure $\NP{u}$ satisfying 
		$$\int\limits_E\NP{u}=\lim\limits_{j\rightarrow\infty}\int\limits_{E\cap\{u>-j\}}\left(\ddc \max\{u,-j\} \right)^n,$$
		for every Borel set $E\subset\Om.$
	\end{definition}
	
	\begin{remark}\label{remk NPMA}
		\begin{itemize}
			\item[i.]If $E\subset\{u>-k\},$ then it follows from \cite[Corollary 4.3]{BT87} that $$\int\limits_E(\ddc\max\{u, -j\})^n=\int\limits_E(\ddc\max\{u,-k\})^n, \text{ for every }j\geq k.$$
			In particular, 
			$$\int_{E\cap\{u>-k\}}\NP{u}=\int\limits_{E\cap\{u>-k\}}(\ddc\max\{u,-k\})^n,$$
			for every $k>0$ and for every Borel set $E\subset\Om$.
			\item[ii.] $\NP{u}$ vanishes on every pluripolar sets.
			\item[iii.] If $\Omega$ is the open unit ball and $u$ is defined by
			$$u(z)=\left(-\log |z_1|\right)^{1/n}(|z_2|^2+...+|z_n|^2-1),$$
			 then $\NP{u}$ is not locally finite (see \cite{Kis84}). 
		\end{itemize}
	\end{remark}
	The following results are classical. We present the proof here for the convenience of the reader.
		\begin{lemma}\label{lem: NP of max}
			Let $u,v \in \pshn$ and $\mu$ be a positive Borel measure that vanishes on pluripolar sets. If $\NP{ u} \geq \mu, \NP{ v} \geq \mu$ then $\NP{\max \{u,v \}} \geq \mu.$
		\end{lemma}
		\begin{proof}
			Since $\mu$ is Borel which does not charge the set $\{ u + v = -\infty\}$, we only need to show that
				$$\int\limits_{E} \NP{\max \{ u,v\}} \geq \int\limits_{E} d\mu,$$
			 for every Borel set $E \subset \{ u+ v >-\infty\}$. Note that $E = \bigcup\limits_{j \geq 1} E_j$ where $E_j = E \cap \{ u + v >-j\}$. We will show that
			\[
			\int\limits_{E_{j_0}} \NP{\max \{ u,v\}} \geq \int\limits_{E_{j_0}} d\mu,
			\]
			for every $j_0\geq 1$.
			
			Since $\max \{u,v\} \geq \min\{u, v\}\geq u+v$, we have 
			$$E_{j_0}\subset\{u+v > -j_0\}\subset \{u>-j\}\cap\{v>-j\} \subset \{\max \{u,v\} > -j \},$$ 
			for every $j > j_0$.  Hence, by Definition \ref{def NPMA} and Remark \ref{remk NPMA} (i), we have
		\begin{equation}\label{eq1lemNPmax}
			\int\limits_{E_{j_0}} \NP{w} = \int\limits_{E_{j_0} } \Big(\ddc \max\{ w, -j \} \Big)^n,
		\end{equation}
			for $w\in\{u, v, \max\{u, v\}\}$ and for every $j>j_0$.
			
			Denote $u_j=\max\{u, -j\}$, $v_j=\max\{v, -j\}$ and $\phi_j=\max\{\max\{u, v\}, -j\}$.
		Observe that $\phi_j = \max \{u_j, v_j\}$. By applying 
			\cite[Proposition 11.9]{Dem89}
			 (see also \cite[Proposition 4.3]{NP09}), we have
			 \begin{equation}\label{eq1.1lemNPmax}
			 (\ddc\phi_j)^n\geq \mathbb{1}_{\{u_j\geq v_j\}}(\ddc u_j)^n+\mathbb{1}_{\{u_j< v_j\}}(\ddc v_j)^n.
			 \end{equation}
			Note $E_{j_0}\cap \{u_j\geq v_j\}=E_{j_0}\cap \{u\geq v\}$ and  $E_{j_0}\cap \{u_j<v_j\}=E_{j_0}\cap \{u< v\}$ for every
			$j>j_0$. Hence, it follows from \eqref{eq1.1lemNPmax} that
			 \begin{equation}\label{eq2lemNPmax}
			 \int\limits_{E_{j_0} } (\ddc\phi_j)^n 
			 \geq \int\limits_{E_{j_0}\cap \{u\geq v\}}(\ddc u_j)^n 
			 +  \int\limits_{E_{j_0}\cap \{u<v\}} (\ddc v_j)^n,
			 \end{equation}
			for every $j>j_0$.
			
			Combining \eqref{eq1lemNPmax} and \eqref{eq2lemNPmax}, we get 
			$$\int\limits_{E_{j_0}} \NP{\max\{u, v\}}\geq
			 \int\limits_{E_{j_0}\cap \{u\geq v\}} \NP{u}+\int\limits_{E_{j_0}\cap \{u<v\}} \NP{v}.$$
			 Thus, by the facts  $\NP{ u} \geq \mu$ and $\NP{v}\geq \mu$, we have
			 	$$\int\limits_{E_{j_0}} \NP{\max\{u, v\}}\geq \int\limits_{E_{j_0}} d\mu.$$
		Letting $j_0\rightarrow\infty$, we obtain 
		$$\int\limits_{E} \NP{\max \{ u,v\}} \geq \int\limits_{E} d\mu.$$
		The proof is completed.
		\end{proof}
	\begin{lemma}\label{NP u+v > NP u}
		Let $u,v\in\pshn$. Then $\NP{(u+v)}\geq\NP{u}+\NP{v}.$
	\end{lemma}
	\begin{proof}
		We  need to show that 
		$$\int\limits_{E} \NP{(u+v)}\geq \int\limits_{E} \NP{u}+ \int\limits_{E} \NP{v},$$
		for every Borel set $E\subset\Om\setminus\{u+v=-\infty\}$.  For $j_0\in\Z^+$, we denote 
		$E_{j_0}=E\cap\{u+v>-j_0\}$. Note that
		\begin{equation}\label{eq0lem NP u+v > NP u}
		E_{j_0}\subset\{u+v>-j\}\subset \{u>-j\}\cap\{v>-j\},
		\end{equation}
		for every $j>j_0$.
		 Hence, by Definition \ref{def NPMA} and Remark \ref{remk NPMA} (i), we have
		 \begin{equation}\label{eq1lem NP u+v > NP u}
		 \int\limits_{E_{j_0}} \NP{w} = \int\limits_{E_{j_0} } \Big(\ddc \max\{ w, -j \} \Big)^n,
		 \end{equation}
		 for $w\in\{u, v, u+v\}$ and for every $j>j_0$. 
		 
		 Denote $u_j=\max\{u, -j\}$, $v_j=\max\{v, -j\}$ and $\phi_j=\max\{u+v, -j\}$. For every  $z\in \{u+v>-j\}$,
		  we have $u_j(z)=u(z)$, $v_j(z)=v(z)$ and $\phi_j(z)=u(z)+v(z)$. Hence
		  $$\phi_j =u_j+v_j,$$
	on the plurifine open set $ \{u+v>-j\}$. Hence, it follows from \cite[Corollary 4.3]{BT87} that
	\begin{equation}\label{eq2lem NP u+v > NP u}
	(\ddc \phi_j)^n|_{\{u+v>-j\}}=(\ddc (u_j+v_j))^n|_{\{u+v>-j\}}\geq \Big((\ddc u_j)^n+(\ddc v_j)^n\Big)|_{\{u+v>-j\}}.
	\end{equation}
	Combining \eqref{eq0lem NP u+v > NP u}, \eqref{eq1lem NP u+v > NP u} and \eqref{eq2lem NP u+v > NP u}, we have
	\begin{align*}
	 \int\limits_{E_{j_0}} \NP{(u+v)}&= \int\limits_{E_{j_0}} (\ddc \phi_j)^n\\
	 & \geq \int\limits_{E_{j_0}}(\ddc u_j)^n+\int\limits_{E_{j_0}}(\ddc v_j)^n\\
	 &= \int\limits_{E_{j_0}} \NP{u}+ \int\limits_{E_{j_0}} \NP{v},
	\end{align*}
	for every $j>j_0$.
	
	Letting $j_0\rightarrow \infty$, we obtain 
	$$ \int\limits_{E} \NP{(u+v)}\geq \int\limits_{E} \NP{u}+ \int\limits_{E} \NP{v}.$$
	The proof is completed.
	\end{proof}
	\section{Stability of subsolutions and supersolutions}\label{secsta}
	The goal of this section is to prove Lemmas \ref{lem: NP of limsup} and \ref{thm NPddcu leq u>-infty mu} which are important tools
	for the proof of the main theorem. First, we need the following lemma:
		\begin{lemma}\label{lemma1}
			Let $u, u_j $ ($j\in\Z^+$) be  negative plurisubharmonic functions on $\Omega$ such that $\{u_j\}_{j\geq 1}$ is monotone and 
			$u=(\lim_{j\to\infty}u_j)^*$.
			 Assume $f, f_j$ are bounded, quasi-continuous on $\Omega$ satisfying $0\leq f,f_j\leq 1$, $f_j$ converges monotonically to $f$ quasi-everywhere. Suppose that $\{f\neq 0\}\subset\{ u\geq -M\}$ and
			 $\{f_j\neq 0\}\subset\{ u_j\geq -M\}$  for every $j$, where $M>0$ is a constant.
			  Then $f_j \NP{u_j}$ converges weakly to $f\NP{u}$ as $j\to\infty$.
		\end{lemma}
		\begin{proof} 
			By the definition, we have
			$$\mathbb{1}_{\{u>-M-1\}}\NP{u}=\mathbb{1}_{\{u>-M-1\}}(\ddc \max\{u,-k\})^n,$$
			for every $k\geq M+1$. Since $\{f\neq0\}\subset\{u>-M-1\}$, it follows that
			\begin{equation}\label{eq1}
		f\NP{u}=f(\ddc\max\{u,-M-1\})^n.
			\end{equation}
		Similar, we also have
			\begin{equation}\label{eq2}
			f_j\NP{u_j}=f_j(\ddc \max \{u_j,-M-1\})^n \text{ for every }j.
			\end{equation}
			Since $u_j$ converges monotonically to $u$, we have $(\ddc \max \{u_j,-M-1\})^n$ converges weakly to 
			$(\ddc\max\{u,-M-1\})^n$ as $j\rightarrow\infty$. Hence, it follows from
			 \cite[Theorem 3.2($4 \Rightarrow 3$)]{BT87} that
			\begin{equation}\label{eq3}
			f_j(\ddc\max\{u_j,-M-1\})^n\stackrel{\text{w}}{\rightarrow} f(\ddc\max\{u,-M-1\})^n.
			\end{equation}
			Combining (\ref{eq1}), (\ref{eq2}) and (\ref{eq3}), we get
			$$f_j \NP{u_j}\stackrel{\text{w}}{\rightarrow} f\NP{u},$$
			as desired.
		\end{proof}
		\begin{lemma}\label{lem: NP of limsup}
			Let $u_j$ be a monotone sequence of negative plurisubharmonic functions on $\Omega$ and let  $\mu$ be a positive Borel measure on $\Om$ such that $\NP{u_j} \geq \mu$ for every $j\in\Z^+$.
			 Assume that  $u := \Big(\lim\limits_{j \to \infty} u_j \Big)^*$ is not identically $-\infty$. Then $\NP{u} \geq \mu$.
		\end{lemma}
		\begin{proof}
			We give the proof for the case where $(u_j)_j$ is increasing. The case of decreasing sequence is similar and we leave it for
			the readers.
			
			For each $k\in\Z^+$, we denote 
			$$f_k=\min\left\lbrace \max\{u_1+k+1,0\},1\right\rbrace .$$
			Then $0\leq f_k\leq 1,$ $f_k|_{\{u_1\geq -k\}}=1$, $f_k|_{\{u_1\leq -k-1\}}=0$ and $f_k$ is continuous in plurifine topology.
			Since $u_1\leq u_2\leq...\leq u_k\leq...\leq u$, we have $\{f_k\neq 0\}\subset \{u_j>-k-1\}\cap\{u>-k-1\}$ for every $j$. Hence,
			it follows from Lemma \ref{lemma1} that
			$f_k\NP{u_j}$ converges weakly to $f_k\NP{u}$ as $j\rightarrow\infty$. Then, by the assumption $\NP{u_j} \geq \mu$, we have
			$$f_k\NP{u}\geq f_k\mu.$$
			Letting $k\rightarrow\infty$, we get
			\begin{equation}\label{eq1lem: NP of limsup}
			\NP{u}\geq \mathbb{1}_{\{u_1>-\infty\}}\mu.
			\end{equation}
			Moreover, the assumption $\NP{u_j} \geq \mu$ implies that $\mu$ vanishes on pluripolar sets. In particular, 
			\begin{equation}\label{eq2lem: NP of limsup}
			\mu=\mathbb{1}_{\{u_1>-\infty\}}\mu.
			\end{equation}
			Combining \eqref{eq1lem: NP of limsup} and \eqref{eq2lem: NP of limsup}, we obtain
			$$	\NP{u}\geq \mu.$$
			The proof is completed.
		\end{proof}
	\begin{lemma}\label{thm NPddcu leq u>-infty mu}
			Let $u_j$ be a monotone sequence of negative plurisubharmonic functions on $\Omega$ such that
			$u := \Big(\lim\limits_{j \to \infty} u_j \Big)^*$ is not identically $-\infty$. 	Let  $\mu$ be a positive Borel measure on $\Om$.
			Assume that there exists a plurifine open subset $U$ of $\Om$ such that 
			$$\mathbb{1}_{U}\NP{u_j}\leq \mu,$$
			for every $j$. Then
			$$\mathbb{1}_{U}\NP{u}\leq \mu.$$
	\end{lemma}
	\begin{proof}
			We give the proof for the case where $(u_j)_j$ is decreasing. The case of increasing sequence is similar and we leave it for	the readers. 
			
			By the quasi-Lindel\"of property of plurifine topology (see Theorem \ref{Lindelof}) and by the fact
			that $\mathcal{B}$ is a basis of plurifine topology, the problem is reduced to the case 
			$U\in\mathcal{B}$, i.e.,
			$$U=\{z\in V: v(z)>0\},$$
			where $V$ is an open subset of $\Om$ and $v$ is a plurisubharmonic function on $V$.
			
			Let $\chi\in C_c(V)$ such that $0\leq\chi\leq 1$ and denote
			$$g_{\chi, k}=\chi \max\{ \min\{4^{k}v-2^k, 1\}, 0\},$$
			for every $k\in\Z^+$. We have $g_{\chi, k}$ is a quasi continuous function on $\Om$.
			
				Denote 
				$$f_k=\min\left\lbrace \max\{u+k+1,0\},1\right\rbrace .$$
				Then $0\leq f_k\leq 1,$ $f_k|_{\{u\geq -k\}}=1$, $f_k|_{\{u\leq -k-1\}}=0$ and $f_k$ is quasi-continuous.
				Since $u_1\geq u_2\geq...\geq u_k\geq...\geq u$, we have $\{f_k\neq 0\}\subset \{u_j>-k-1\}\cap\{u>-k-1\}$ for every $j$. Hence, it follows from Lemma \ref{lemma1} that
				$f_kg_{\chi, k}\NP{u_j}$ converges weakly to $f_kg_{\chi, k}\NP{u}$ as $j\to\infty$.
				Moreover, since $\mbox{supp} g_{\chi, k}\subset U$ and $0\leq f_k, g_{\chi, k}\leq 1$, we have
				 $f_kg_{\chi, k}\NP{u_j}\leq\mu$ for every $j$. Then
				 $$f_kg_{\chi, k}\NP{u}\leq \mu.$$
				 Letting $k\to\infty$ and $\chi\nearrow \mathbb{1}_V$, we get
				 $$\mathbb{1}_U\NP{u}\leq \mu.$$
				 The proof is completed.
	\end{proof} 
	\section{An envelope of plurisubharmonic functions}\label{secenvelope}
	The main result of this section is as follows:
	\begin{theorem}\label{the NPMA of envelope}
		Let $\mu$ be a positive Borel measure on $\Om$ and let $U\subset\Om$ be a plurifine open set such that
		$$\NP{\varphi}\geq \mathbb{1}_U\mu,$$
		for some $\varphi\in\pshn$. Denote $$u=\left(\sup\{w\in\pshn: w\leq H \mbox{ on } \Om\setminus U, \NP{w}\geq \mathbb{1}_U\mu \}\right)^*,$$
		where $H$ is a negative plurisubharmonic function on $\Om$.
		Then $\mathbb{1}_U\NP{u}=\mathbb{1}_U\mu.$
	\end{theorem}
		In order to prove the above theorem, we need the following lemmas:
			\begin{lemma}\label{cor envelop}
				Let $S$ be a family of negative plurisubharmonic functions on $\Omega$ and let  $\mu$ be a positive Borel measure on $\Om$ such that
				$\NP{w} \geq \mu$ for every $w\in S$. Denote
				$$u_S=\left(\sup\{w: w\in S\}\right)^*.$$
				Then $\NP{u_S}\geq\mu$.
			\end{lemma}
			\begin{proof}
				By Choquet's lemma \cite[Lemma 2.3.4]{Kli91}, there exists a sequence $\{u_j\}_{j\in\Z^+}\subset S$ such that
				$$u_S=\left(\sup\{u_j: j\in\Z^+\}\right)^*.$$
				For every $j\in\Z^+$, we denote
				$$v_j=\max\{u_1, u_2,..., u_j\}.$$
				Then $\{v_j\}$ is an increasing sequence  and  $u_S=(\lim_{j\to\infty}v_j)^*$. Moreover, it follows from Lemma \ref{lem: NP of max} that for every $j \in \Z^+$,
				$$\NP{v_j}\geq\mu.$$
				 Hence, by Lemma \ref{lem: NP of limsup}, we have 
				$$\NP{u_S}\geq\mu.$$
			\end{proof}
				\begin{lemma}\label{lem upper reg}
					Let $S$ be a family of negative plurisubharmonic functions on $\Omega$. Assume that there exist a set $W\subset\Om$ and a function
					$H: W\rightarrow\R$ such that $w|_W\leq H$ for every $w\in S$. Put
					$$u_S=(\sup\{w: w\in S\})^*.$$
					Then, there exists a pluripolar set $N\subset \Om$ such that $u_S\leq H$ on $W\setminus N$. 
				\end{lemma}
				\begin{proof}
					Set $v_S=\sup\{w: w\in S\}$. Since negligible sets are pluripolar, we have $\{u_S>v_S\}$ is pluripolar. 
					By Josefson's theorem, there exists $\psi \in \pshn$ such that $\{u_S>v_S\} \subset \{ \psi = -\infty\}$. 
					Therefore, for all $\epsilon > 0,$
					$$u_S + \epsilon \psi \leq v_S.$$
					Since $v_S\leq H$ on $W$, it follows that $u_S + \epsilon \psi \leq H$ on $W$. Letting $\epsilon\searrow 0$, we get
					$u_S\leq H$ on $W\setminus\{\psi=-\infty\}$.
					
					The proof is completed.
				\end{proof}
				\begin{lemma}\label{lemAhag}
					Let $u$ be a bounded, negative plurisubharmonic function on $\Om$ and let $D\Subset\Om$ be an open ball. Denote by $u_D$
					 the smallest maximal plurisubharmonic majorant of $u$ in $D$. Assume that  
					 $\mu$ is a non-pluripolar positive Radon measure on $D$ such that $\mu(D)<+\infty$. Then, there exists
					  $v\in\mathcal{F}(D, u_D)$ such that $(dd^c v)^n=\mu$. Here, $\mathcal{F}(D, u_D)$ is the set of plurisubharmonic functions $\varphi$ on $D$ satisfying $u_D+w\leq \varphi\leq u_D$ for some $w\in\mathcal{F}(D)$.
				\end{lemma}
				\begin{proof}
					This lemma is an immediate corollary of \cite[Theorem 3.7]{ACCP}.  Here we will give a proof that does not use
					 \cite[Theorem 3.7]{ACCP}.
					 
					 Let $u_j$ be a sequence of smooth plurisubharmonic functions decreasing to $u$ on a neighborhood $V$ of $\overline{D}$.
					 It is classical that for every $j$, there exists a unique maximal plurisubharmonic function $u_{j, D}$ on $D$ 
					 such that  $\lim_{D\ni z\to z_0}u_{j, D}(z)=u_j(z_0)$ for every $z_0\in\partial D$. It is easy to check that
					 $u_{j, D}$ is decreasing to $u_D$ as $j\rightarrow\infty$.
		
					  By \cite[Theorem 3.4]{Ahag07}, there exists a unique $v_j\in \mathcal{F}(D, u_{j, D})$ such that $(dd^c v_j)^n=\mu$.
					 Moreover, it follows from the comparison principle \cite[Theorem 3.2]{Ahag07} that $v_j$ is a decreasing sequence and
					 $$v_0+ u_{j, D}\leq v_j\leq  u_{j, D},$$
					 where $v_0$ is the unique function in $\mathcal{F}(D)$ satisfying $(dd^cv_0)^n=\mu$.
					 Denote $v=\lim_{j\to\infty}v_j$. We have $v_0\leq u_D\leq v\leq u_D$ and $(dd^c v)^n=\mu$.
					 
					This finishes the proof.
					\end{proof}
			Now we begin to prove Theorem \ref{the NPMA of envelope}.
			We first consider the case where $H$ is bounded and $\mu(\Om)<+\infty$.
				\begin{theorem}\label{the: NPMA eq on plurifine open}
					Let $\mu$ be a non-pluripolar positive Radon measure on $\Omega$ such that $\mu(\Om)<+\infty$. Let $U$ be a plurifine open subset of $\Om$ and $H \in L^\infty (\Omega)$. Denote
					\[
					u = \Big( \sup \{ w\in \pshn \colon w \leq H \text{ on } \Omega \setminus U, \NP{w} \geq \mathbb{1}_U \mu \} \Big)^*.
					\]
					Then $u\in\mathcal{D}(\Om)$ and $\mathbb{1}_U (dd^c u)^n = \mathbb{1}_U \mu. $
				\end{theorem}
		\begin{proof}
			We first show that $u$ is well-defined, i.e., the family
			$$S:= \{ w\in \pshn \colon w \leq H \text{ on } \Omega \setminus U, \NP{w} \geq \mathbb{1}_U \mu\}$$
			 is non-empty. 
			 
			 Let $D$ be an open ball containing $\overline{\Omega}$. By \cite[Theorem 5.14]{Ceg04}, there exists $\varphi\in\F(D)$ such that
		  $(\ddc \varphi)^n = \mathbb{1}_U\mu.$ Put $M=-\inf_{\Om}H$. We have $\varphi|_{\Om}-M\in S$. Hence,
		 $u$ is well-defined. Moreover, since $\varphi\in\F(D)$ and  $\varphi|_{\Om}-M\leq u$, we have $u\in\mathcal{D}(\Om)$.

			 It remains to show that $\mathbb{1}_U (dd^c u)^n = \mathbb{1}_U \mu$. We first consider the case where $U$ is open in
			 the usual topology. In this case, we only need to show that $\mathbb{1}_{B} (\ddc u)^n = \mathbb{1}_{{B}} \mu$ for any open ball ${B} \Subset U$.
			
			 Since $\mu(\Om)<\infty$, without loss of generality, we can assume that $\mu (\partial B)=0$. Set
			\[
			u_B = \Big(\sup \{w \in \pshn \colon w \leq u \text{ on } \Omega \setminus B \} \Big)^*.
			\]
			Then, $u_B$ is maximal on $B$. In particular, $\mathbb{1}_B(\ddc u_B)^n = 0 \leq \mathbb{1}_B\mu.$ By  
			Lemma \ref{lemAhag}, there exists  $w_B \in \mathcal{N}(B,u_B)$ such that 
			\[(\ddc w_B)^n = \mathbb{1}_B \mu.\]
			Here, the notation $w_B \in \mathcal{N}(B,u_B)$ means that there exists a function $\psi \in \mathcal{N}(B)$
			( $\mathcal{N}(B)$ is the set of functions belonging in $\mathcal{D}(B)$ with smallest maximal plurisubharmonic majorant identically zero) such that 
			\[u_B + \psi \leq w_B \leq u_B\quad\mbox{on } B.\]
			
			By Lemma \ref{cor envelop}, we have $(dd^cu)^n\geq \mathbb{1}_U\mu$. Then, it follows from the comparison principle 
			\cite[Theorem 3.1]{ACCP} (see also \cite[Corollary 3.2]{ACCP}) that $u|_B\leq w_B$. On the other hand, for every $z_0\in\partial B$, we have
			$$\limsup\limits_{B\ni z\to z_0}w_B(z)\leq \limsup\limits_{B\ni z\to z_0}u_B(z)
			=\limsup\limits_{U\setminus B\ni z\to z_0}u_B(z)=\limsup\limits_{U\setminus B\ni z\to z_0}u(z)
				=u(z_0).$$
			Hence, the  function
			\[
			\overline{u}_B := \begin{cases}
			u \text{ on } \Omega \setminus B,\\
			w_B \text{ on } B,
			\end{cases}
			\]
			is plurisubharmonic on $\Om$. Moreover,
			$$(dd^c \overline{u}_B)^n\geq \mathbb{1}_B (dd^c w_B)^n+\mathbb{1}_{\Om\setminus\bar{B}} (dd^c u)^n\geq 
			 \mathbb{1}_{U\setminus\partial B}\mu= \mathbb{1}_{U}\mu.$$
			Hence, $\overline{u}_B\in S$. Consequently, $\overline{u}_B \leq u$ on $\Om$. Recall that $w_B \geq u$ on $B$, thus $\overline{u}_B \geq u$ on $\Om$. Then $\overline{u}_B = u$ on $\Om$ and it follows that
			$$(dd^c u)^n|_B=(dd^c \overline{u}_B)^n|_B=(dd^c w_B)^n=\mathbb{1}_{B}\mu.$$
			
			Now, we consider the general case where $U$ is plurifine open. 	By Lemma \ref{open subsets decreasing to plurifine open set}, there exists a decreasing sequence of open subset $U_j$ of $\Om$ such that $U\subset  \bigcap\limits_{j \geq 1} U_j$ and
			$\bigcap\limits_{j \geq 1} U_j\setminus U$ is  pluripolar. For every $j$, we denote
				$$S_j= \{ w\in \pshn \colon w \leq H \text{ on } \Omega \setminus U_j, \NP{w} \geq \mathbb{1}_U \mu\},$$
			and
			$$u_j=\left(\sup\{w: w\in S_j\}\right)^*.$$
			By using the case where $U$ is open, we have $u_j\in\mathcal{D}(\Om)$ and
			\begin{equation}\label{eq1thefineopen}
			\mathbb{1}_{U_j} (dd^c u_j)^n = \mathbb{1}_U \mu.
			\end{equation}
		Since $U\subset U_{j+1}\subset U_j$ for every $j$, we have $u_j$ is a decreasing sequence and $u_j\geq u$ for every $j$. Hence
			\begin{equation}\label{eq2thefineopen}
			\bar{u}:=\lim_{j\to\infty}u_j\geq u.
			\end{equation}
			Moreover, using \eqref{eq1thefineopen} and applying Lemmas \ref{lem: NP of limsup} and \ref{thm NPddcu leq u>-infty mu}
			(replace $\mu$ by $\mathbb{1}_U\mu$), we get
			\begin{equation}\label{eq3thefineopen}
			\mathbb{1}_{U} (dd^c \bar{u})^n = \mathbb{1}_U \mu.
			\end{equation}
			It remains to show that $u=\bar{u}$. By Lemma \ref{lem upper reg}, for every $j$, there exists a pluripolar set $N_j$ such that
			$u_j\leq H$ on $\Om\setminus (U_j\cup N_j)$. Denote $N=\cup_{j=1}^{\infty}N_j$. We have $N$ is pluripolar and 
			$\bar{u}\leq H$ on $\Om\setminus (U\cup N)$. By Josefson's theorem, there exists a negative plurisubharmonic function $\psi$
			 on $\Om$ such that $N\subset\{\psi=-\infty\}$. Then, we have
			 $\bar{u}+\epsilon\psi\leq H$ on $\Om\setminus U$ and, by Lemma \ref{NP u+v > NP u}, 
			 $$\NP{(\bar{u}+\epsilon\psi)}\geq (dd^c\bar{u})^n\geq \mathbb{1}_U\mu.$$
			 Hence $\bar{u}+\epsilon\psi\in S$ for every $\epsilon>0$. As a consequence, we have
			 \begin{equation}\label{eq4thefineopen}
			  u\geq\left(\limsup_{\epsilon\to 0}(\bar{u}+\epsilon\psi)\right)^*=\bar{u}
			 \end{equation}
			Combining \eqref{eq2thefineopen} and \eqref{eq4thefineopen}, we get $\bar{u}=u$. Therefore, by \eqref{eq3thefineopen},
			we obtain $\mathbb{1}_{U} (dd^c u)^n = \mathbb{1}_U \mu$.
			
			The proof is completed.
		\end{proof}

	\begin{proof}[End of the proof of Theorem \ref{the NPMA of envelope}]
	For every $j, k\in\Z^+$, we denote
	$$U_j=\{z\in U: d(z, \partial\Omega)>2^{-j}, \varphi (z)>-2^j \},$$
	and
	$$H_k=\max\{H, -k\}.$$
	We also define
	$$u_{j, k}=\left(\sup\{w\in\pshn: w\leq H_k \mbox{ on } \Om\setminus U, \NP{w}\geq \mathbb{1}_{U_j}\mu \}\right)^*,$$
	and
		$$u_{j}=\left(\sup\{w\in\pshn: w\leq H \mbox{ on } \Om\setminus U, \NP{w}\geq \mathbb{1}_{U_j}\mu \}\right)^*.$$
	It is clear that the sequence $\{u_{j, k}\}_{k\in\Z^+}$ is decreasing for every $j$ and 
	\begin{equation}\label{eq1theNPMAenvelope}
	u_{j, k}\geq u_j,
	\end{equation}
 for every $j, k$. The assumption $\NP{\phi} \geq \mathbb{1}_{U} \mu$ implies that $\int_{U_j}d\mu<\infty$. It follows from Theorem \ref{the: NPMA eq on plurifine open} that
 $$\mathbb{1}_U\NP{u_{j, k}}=\mathbb{1}_{U_j}\mu,$$
for every $j, k$. Letting $k\to\infty$ and using Lemmas \ref{lem: NP of limsup} and \ref{thm NPddcu leq u>-infty mu}, we get
\begin{equation}\label{eq2theNPMAenvelope}
\mathbb{1}_U\NP{\bar{u}_j}=\mathbb{1}_{U_j}\mu,
\end{equation}
where $\bar{u}_j=\lim_{k\to\infty}u_{j, k}$. 

Now we will prove $\bar{u}_j=u_j$. By Lemma \ref{lem upper reg}, for every $k$, there exists a pluripolar set $N_{j, k}$ such that
$u_{j, k}\leq H_k$ on $\Om\setminus (U\cup N_{j, k})$. Denote $N_j=\cup_{k=1}^{\infty}N_{j, k}$. We have $N_j$ is pluripolar and 
$\bar{u}_j\leq H$ on $\Om\setminus (U\cup N_j)$. By Josefson's theorem, there exists a negative plurisubharmonic function $\psi_j$
on $\Om$ such that $N_j\subset\{\psi_j=-\infty\}$. Then, we have
$\bar{u}+\epsilon\psi_j\leq H$ on $\Om\setminus U$ and, by Lemma \ref{NP u+v > NP u}, 
$$\NP{(\bar{u}_j+\epsilon\psi_j)}\geq \NP{\bar{u}_j}\geq \mathbb{1}_{U_j}\mu.$$
By the definition of $u_j$, we have $\bar{u}_j+\epsilon\psi_j\leq u_j$ for every $\epsilon>0$. Hence
\begin{equation}\label{eq3theNPMAenvelope}
\bar{u}_j=\left(\lim_{\epsilon\to 0}(\bar{u}_j+\epsilon\psi_j)\right)^*\leq u_j.
\end{equation}
Combining \eqref{eq1theNPMAenvelope} and \eqref{eq3theNPMAenvelope}, we get $\bar{u}_j=u_j$. Then, by \eqref{eq2theNPMAenvelope}, we have
$$\mathbb{1}_U\NP{u_j}=\mathbb{1}_{U_j}\mu.$$
 Letting $j\to\infty$ and using Lemmas \ref{lem: NP of limsup} and \ref{thm NPddcu leq u>-infty mu}, we get
 $\mathbb{1}_{U_{j_0}}\NP{\bar{u}}=\mathbb{1}_{U_{j_0}}\mu$
 for every $j_0\in\Z^+$, where $\bar{u}=\lim_{j\to\infty}u_j$. By the same argument as above, we also have $\bar{u}=u$.  Hence
 \begin{equation}\label{eq4theNPMAenvelope}
 \mathbb{1}_{U_{j_0}}\NP{u}=\mathbb{1}_{U_{j_0}}\mu,
 \end{equation}
 for every $j_0\in\Z^+$.
 Observe that $\cup_{j_0=1}^{\infty}U_{j_0}=U\setminus\{\phi=-\infty\}$ and $\mu(\{\phi=-\infty\})=0$.  Hence, by using \eqref{eq4theNPMAenvelope} and letting $j_0\rightarrow\infty$, we have
 $$\mathbb{1}_{U}\NP{u}=\mathbb{1}_{U}\mu.$$
 	This finishes the proof.
		\end{proof}
	The following result is a corollary of Theorem \ref{the NPMA of envelope} and Lemma \ref{thm NPddcu leq u>-infty mu}:
	\begin{proposition}\label{propmodel}
		Let $u\in\pshn$. Then there exists $\bar{u}\in\pshn$ such that $u\leq\bar{u}\leq P[u]$ and $\NP{\bar{u}}=0$.
		In particular, if $u$ is model then $\NP{u}=0$.
	\end{proposition}
	\begin{proof} For every $j\in\Z^+$, we denote
			$$V_j=\{z\in\Omega:\ d(z,\partial\Omega)>2^{-j},\ u(z)>-2^j\},$$
		and
			$$u_j=\left(\sup\left\{v\in\psh:\ v\leq u \text{ on }\Omega\setminus V_j\right\}\right)^*.$$
	By using Theorem \ref{the NPMA of envelope}, we have $\mathbb{1}_{V_j}\NP{u_j}=0$ for every $j$. Then, it follows from 
	Lemma \ref{thm NPddcu leq u>-infty mu} that $\mathbb{1}_{V_j}\NP{\bar{u}}=0$ for every $j$, where $\bar{u}=(\lim_{j\to\infty}u_j)^*$.
	Since $\cup_{j=1}^{\infty}V_j=\Om\setminus\{u=-\infty\}$ and $\NP{\bar{u}}$ vanishes on pluripolar sets, it follows that 
	$$\NP{\bar{u}}=0.$$
	Moreover, since $u\leq u_j\leq P[u]$ for every $j$, we also have $u\leq\bar{u}\leq P[u]$.
	
	The proof is completed.
	\end{proof}
	\section{Xing-type comparison principles}\label{sec: comparison}
	In \cite{Xin}, Xing provided a strong comparison principle for bounded plurisubharmonic functions. Xing's theorem then has been generalized	by Nguyen-Pham \cite{NP09} and by \AA hag-Cegrell-Czy\.z-Pham \cite{ACCP}. In this section, we introduce two new  Xing-type theorems
	(Theorems \ref{lem: compa 1} and  \ref{strong CP}) and some applications.
		\begin{theorem}\label{lem: compa 1}
			Let $u,v\in\pshn$ such that
			\begin{itemize}
				\item[i,]$\liminf\limits_{\Omega\setminus N\ni z\rightarrow\xi_0}(u(z)-v(z))\geq0$ for every $\xi_0\in\partial\Omega$,
				where $N=\{v=-\infty\}$;
				\item[ii,] $v\leq u+O(1)$ on $\Omega$.
			\end{itemize}
			Let  $w_j\in\text{PSH}(\Omega, [-1, 0]),$ $j=1,...,n,$ and denote $T=dd^cw_1\wedge...\wedge dd^c w_n$. Then
			\begin{equation}\label{eq01lemcompa}
			\frac{1}{n!}\int\limits_{\{u<v\}}(v-u)^nT +\int\limits_{\{u<v\}}(-w_1)\NP{v}	\leq
			\int\limits_{\{u<v\}}(-w_1)\NP{u}.
			\end{equation}
				Moreover, if $\NP{u} \leq \NP{v} + \mu$ for some positive Borel measure $\mu$  then
				\begin{equation}\label{eq02lemcompa}
				\frac{1}{n!}\int\limits_{\{u<v\}}(v-u)^n T \leq \int\limits_{\{u<v\}}(-w_1)d\mu.
				\end{equation}
		\end{theorem}
		\begin{proof} For each $M>0$, we denote
			$$u_M=\max\{u,-M\}\quad\mbox{and}\quad v_M=\max\{v,-M\}.$$ 
			By the assumption $(i)$, we have
			$$\liminf\limits_{\Omega\setminus N\ni z\rightarrow\xi_0}(u_M(z)-(1+\epsilon)v_M(z))
		\geq	\liminf\limits_{\Omega\setminus N\ni z\rightarrow\xi_0}(u_M(z)-v_M(z))\geq 0,$$ 
			for every $\xi_0\in\partial\Omega$ and $M\in\Z^+$.
			Hence, by using \cite[Theorem 4.9]{NP09} (observe that, in this theorem, the condition {\it $\Om$ is hyperconvex}
			 is not	necessary), we have
			 \begin{equation}\label{eq1lemcompa}
			 	\frac{1}{n!}\int\limits_{E_M}((1+\epsilon)v_{M}-u_M-\epsilon)^nT
			 	\leq\int\limits_{E_M}(-w_1)\left((dd^c u_M)^n-(dd^c ((1+\epsilon)v_{M}))^n\right),
			 \end{equation}
			 where $E_M=\{u_M<(1+\epsilon)v_{M}-\epsilon\}\Subset\Om$.
			 
			 Note that if $z\in E_M$ then $v(z)>-\frac{M}{1+\epsilon}$ and $v(z)> (1+\epsilon)v(z)>u_M(z)\geq u(z)$.
			 Moreover, by the assumption $(ii)$, there exists $K>1$ such that $v\leq u+K$. Hence, we have
			 $$E_M\subset  \left\{v>-\frac{M}{1+\epsilon}\right\}\cap \{u<v\} \subset \{u>-M\}\cap \{u<v\},$$
			 for every $M\geq \frac{(1+\epsilon)K}{\epsilon}$.
			  In particular,
			  $$\mathbb{1}_{E_M}(dd^c u_{M})^n=\mathbb{1}_{E_M}\NP{u}\leq \mathbb{1}_{\{u<v\}}\NP{u},$$
			  for every $M\geq \frac{(1+\epsilon)K}{\epsilon}$.
			  Hence, 
			 \begin{equation}\label{eq2lemcompa}
			 \int\limits_{E_M}(-w_1)(dd^c u_{M})^n\leq 
			  \int\limits_{\{u<v\}}(-w_1)\NP{u},
			 \end{equation}
			 for $M\gg 1$.
			 
			By the fact $E_M\subset\{v>-M\}$, we also have
			\begin{equation}\label{eq3lemcompa}
			\int\limits_{E_M}(-w_1)(dd^c ((1+\epsilon)v_{M}))^n
			= \int\limits_{E_M}(-w_1)\NP{(1+\epsilon)v}\geq \int\limits_{E_M}(-w_1)\NP{v}.
			\end{equation}
			Combining \eqref{eq1lemcompa}, \eqref{eq2lemcompa} and \eqref{eq3lemcompa}, we get
			
$$\frac{1}{n!}\int\limits_{E_M}((1+\epsilon)v_{M}-u_M-\epsilon)^nT+\int\limits_{E_M}(-w_1)\NP{v}\leq\int\limits_{\{u<v\}}(-w_1)\NP{u},$$
for every $M\gg 1$.
		
			Letting $M\rightarrow\infty$ and using the monotone convergence theorem (observer that 
			$\{\mathbb{1}_{E_M}((1+\epsilon)v_{M}-u_M-\epsilon)^n\}_{M\in\Z^+}$ is an increasing sequence), we have
			$$\frac{1}{n!}\int\limits_{\{u<(1+\epsilon)v-\epsilon\}}((1+\epsilon)v-u-\epsilon)^nT
			+\int\limits_{\{u<(1+\epsilon)v-\epsilon\}}(-w_1)\NP{v}
			\leq \int\limits_{\{u<v\}}(-w_1)\NP{u}.$$
			Letting $\epsilon\searrow 0$, we obtain the inequality \eqref{eq01lemcompa}.
			
			It remains to prove \eqref{eq02lemcompa}. For every $M\geq \frac{(1+\epsilon)K}{\epsilon}$, by the fact $E_M\subset\{u>-M\}\cap\{v>-M\}$, we have
			$$\left((dd^c u_M)^n-(dd^c ((1+\epsilon)v_{M}))^n\right)|_{E_M}=(\NP{u}-\NP{((1+\epsilon)v)})|_{E_M}\leq\mu|_{E_M}.$$
			Hence, it follows from \eqref{eq1lemcompa} that
			$$	\frac{1}{n!}\int\limits_{E_M}((1+\epsilon)v_{M}-u_M-\epsilon)^nT
			\leq\int\limits_{E_M}(-w_1) d\mu,$$
			for every $M\gg 1$. Letting $M\rightarrow\infty$, we get
				$$	\frac{1}{n!}\int\limits_{\{u<(1+\epsilon)v-\epsilon\}}((1+\epsilon)v-u-\epsilon)^nT
				\leq\int\limits_{\{u<(1+\epsilon)v-\epsilon\}}(-w_1) d\mu.$$
				Letting $\epsilon\searrow 0$, we obtain \eqref{eq02lemcompa}. This finishes the proof.
		\end{proof}
		 \begin{corollary}\label{cor: comparison 1}
		 		Let $u,v\in\pshn$ such that
		 		\begin{itemize}
		 			\item[i,]$\liminf\limits_{\Omega\setminus N\ni z\rightarrow\xi_0}(u(z)-v(z))\geq0$ for every $\xi_0\in\partial\Omega$,
		 			where $N=\{v=-\infty\}$;
		 			\item[ii,] $v\leq u+O(1)$ on $\Omega$.
		 		\end{itemize}
		 		 If $\NP{u} \leq \NP{v}$ then $u \geq v$. 
		 \end{corollary}
		 \begin{proof} By the last assertion of Theorem \ref{lem: compa 1}, we have
		 	$$\int\limits_{\{u<v\}}(v-u)^n(dd^cw)^n=0,$$
		 	for every $w\in\text{PSH}(\Omega, [-1, 0])$. It follows that $v\leq u$ a.e., and thus $v\leq u$ everywhere in $\Om$.
		 \end{proof}
			\begin{corollary}\label{cor NP=0 implies model}
				If $u\in\pshn$ and $\NP{u}=0$ then $u$ is model.
			\end{corollary}
			\begin{proof}
				Let $v\in\pshn$ such that $v\leq u+ O(1) \text{ on }\Omega$ and $ \liminf\limits_{\Omega\setminus N\ni z\rightarrow\xi_0}(u(z) - v(z))\geq0$ for every  $\xi_0\in\partial\Omega$, where $N=\{u=-\infty\}\subset\{v=-\infty\}$.
				 For $\epsilon>0,$ we denote $$v_\epsilon(z)=v(z)+\epsilon(\|z\|^2-M),$$ where $M=\sup\limits_{\overline{\Omega}}\|z\|^2.$ Then $v_\epsilon\in\pshn$, $v_\epsilon\leq u+ O(1) \text{ on }\Omega$ and $ \liminf\limits_{\Omega\setminus N\ni z\rightarrow\xi_0}(u(z) - v_\epsilon(z)\geq0$ for every  $\xi_0\in\partial\Omega$.
				It follows from Corollary \ref{cor: comparison 1} that 
			 $u\geq v_{\epsilon}$ for every $\epsilon>0$. Letting $\epsilon\rightarrow 0$, we obtain $u\geq v$. Taking the supremum over all such $v$ and taking the upper semi-continuous regularization yields $u\geq P[u]$ almost everywhere in $\Omega$, hence $u\geq P[u]$ everywhere in $\Omega$. It is clear that $u\leq P[u]$. Therefore $u=P[u]$, which means $u$ is model, as desired.
			\end{proof}
		\begin{theorem}\label{the: model}
			Suppose $u\in\pshn$. Then 
			\begin{itemize}
				\item [(i)] $P[u]$ is a model plurisubharmonic function;
				\item [(ii)] $u$ is model iff $\NP{u}=0$.
			\end{itemize}
		\end{theorem}
	\begin{proof}
		(ii) is an immediate corollary of Proposition \ref{propmodel} and Corollary \ref{cor NP=0 implies model}.
		It remains to prove (i).
		
		 By Proposition \ref{propmodel}, there exists $\bar{u}\in\pshn$ such that $u\leq\bar{u}\leq P[u]$ and
		$\NP{\bar{u}}=0$. Then, by Corollary \ref{cor NP=0 implies model}, we have $\bar{u}$ is model and it 
		follows that
		$$u\leq\bar{u}=P[\bar{u}]\leq P[u].$$
		Moreover, the condition $u\leq \bar{u}$ implies that $P[u]\leq P[\bar{u}]$. Hence
		$$\bar{u}=P[\bar{u}]=P[u].$$
		Thus, $P[u]$ is model.
	\end{proof}
			\begin{remark}\label{rmkF}
				\begin{itemize}
					\item [(i)]	If $u$ is a negative maximal plurisubharmonic function then it follows directly from the definitions that $u$ is model. However, the converse is not true. For example, $u=\log |z|$ is a model plurisubharmonic function which is not maximal on the unit ball.
					\item [(ii)] Let $u\in\pshn\cap \mathcal{D}(\Om)$ such that $(dd^cu)^n$ vanishes on pluripolar sets. If
					$v\geq u$ is a model plurisubharmonic function then it follows
					from \cite[Theorem 1.2]{Blo06} and \cite[Lemma 4.1]{ACCP}  that $v\in\mathcal{D}(\Om)$ and
					 $(dd^cv)^n$ vanishes on pluripolar sets. Hence, by Theorem \ref{the: model}, we have $(dd^cv)^n=0$, i.e.,
					 $v$ is maximal. Consequently, if $\Om$ is a hyperconvex domain then
					 $$\{w\in \mathcal{F}(\Om): (dd^c w)^n \mbox{ vanishes on pluripolar sets}\}\subset\NNP(\Om).$$
				\end{itemize}
			\end{remark}
			\begin{theorem}\label{strong CP}
				Let $u,v,H\in\pshn$ such that $u\in\NNP(H)$ and $v\leq H$. 
					Assume $w_j\in\text{PSH}(\Omega,[-1, 0]),$ $j=1,...,n,$ and denote $T=dd^cw_1\wedge...\wedge dd^c w_n$. Then
				$$\frac{1}{n!}\int\limits_{\{u<v\}}(v-u)^n T+\int\limits_{\{u<v\}}(-w_1)\NP{v}	\leq
				\int\limits_{\{u<v\}}(-w_1)\NP{u}.$$
					Moreover, if $\NP{u} \leq \NP{v} + \mu$ for some positive Borel measure $\mu$  then
					$$\frac{1}{n!}\int\limits_{\{u<v\}}(v-u)^n T \leq \int\limits_{\{u<v\}}(-w_1)d\mu.$$
			\end{theorem}
			\begin{proof}
				We will use the same idea as in the proof of \cite[Theorem 3.1]{ACCP}. Recall that
				$$\mathcal{N}_{NP}(H)=\{w\in\pshn: \mbox{there exists } v\in\mathcal{N}_{NP} \mbox{ such that } v+H\leq w\leq H \},$$
				where $\mathcal{N}_{NP}$ is the set of negative plurisubharmonic functions $u$ satisfying $P[u]=0$.
				
				Let $\varphi\in\NNP$ such that $H\geq u\geq H+\varphi$.
				For every $j\in\Z^+$, we denote
				$$V_j=\{z\in\Omega:\ d(z,\partial\Omega)>2^{-j},\ \varphi(z)>-2^j,\ H(z)>-2^j\},$$
				and
				$$\varphi_j=(\sup\{\psi\in\pshn:\ \psi\leq\varphi\text{ on }\Omega\setminus V_j\})^*.$$
				Since $v\leq H$, we have, for every $j\in\Z^+$,
				\begin{equation}\label{Omega-V_j}
				u\geq H+\varphi=H+\varphi_j\geq\varphi_j+v \text{ on } \Omega\setminus \overline{V_j},
				\end{equation}
				which implies
				\begin{equation}\label{liminf partial Omega}
				\liminf\limits_{\Om\setminus N\ni z\to\partial\Omega}[u-(\varphi_j+v)]\geq 0,
				\end{equation}
			where $N=\{u=-\infty\}\subset\{\varphi+v=-\infty\}$.
				We also have,
				\begin{equation}\label{V_j}
				u\geq H+\varphi\geq -2^{j+1}\geq (\varphi_j+v)-2^{j+1} \text{ on } \overline{V_j}.
				\end{equation}
				By the inequalities (\ref{Omega-V_j}) and (\ref{V_j}), we have
				\begin{equation}\label{same singular}
				\varphi_j+v\leq u+2^{j+1} \text{ on }\Omega.
				\end{equation}
				By using the inequalities (\ref{liminf partial Omega}) and (\ref{same singular}), and applying Theorem \ref{lem: compa 1},
				 we have
				\begin{align*}
				\frac{1}{n!}\int\limits_{\{u<\varphi_j+v\}}(\varphi_j+v-u)^nT+&\int\limits_{\{u<\varphi_j+v\}}(-w_1)\NP{(\varphi_j+v)}\\
				\leq
				&\int\limits_{\{u<\varphi_j+v\}}(-w_1)\NP{u}.
				\end{align*}
				Hence, by Lemma~\ref{NP u+v > NP u}, we obtain
				$$\frac{1}{n!}\int\limits_{\{u<\varphi_j+v\}}(\varphi_j+v-u)^nT
				+\int\limits_{\{u<\varphi_j+v\}}(-w_1)\NP{v}
				\leq
				\int\limits_{\{u<\varphi_j+v\}}(-w_1)\NP{u}.$$
				Then, by the monotone convergence theorem, we have
				\begin{align*}
				\frac{1}{n!}\int\limits_{\{u<\lim\limits_{j\rightarrow\infty}\varphi_j+v\}}(\lim\limits_{j\rightarrow\infty}\varphi_j+v-u)^nT+&\int\limits_{\{u<\lim\limits_{j\rightarrow\infty}\varphi_j+v\}}(-w_1)\NP{v}\\
				\leq
				&\int\limits_{\{u<\lim\limits_{j\rightarrow\infty}\varphi_j+v\}}(-w_1)\NP{u}.
				\end{align*}
				By the same argument as in the proof of Proposition \ref{propmodel}, we have $\NP{(\lim\limits_{j\rightarrow\infty}\varphi_j)^*}=0$, and then it follows from Corollary
				\ref{cor NP=0 implies model} that $(\lim\limits_{j\to\infty}\varphi_j)^*$ is model. 
				Hence, by the condition $\varphi\in\NNP$ and the fact $\varphi\leq\varphi_j$ for every $j$, we have $(\lim\limits_{j\rightarrow\infty}\varphi_j)^*=0$ and hence $\lim\limits_{j\rightarrow\infty}\varphi_j=0$ outside a pluripolar set.
				It thus follows that
				$$\frac{1}{n!}\int\limits_{\{u<v\}}(v-u)^nT+\int\limits_{\{u<v\}}(-w_1)\NP{v}
				\leq
				\int\limits_{\{u< v\}}(-w_1)\NP{u}.$$
				Now, assume that $\NP{u} \leq \NP{v} + \mu$. Since $\NP{v}\leq\NP{(\varphi_j+v)}$, we have 
				$\NP{u} \leq \NP{(\varphi_j+v)} + \mu$. 
					By using the inequalities (\ref{liminf partial Omega}) and (\ref{same singular}),
					 and applying Theorem \ref{lem: compa 1},
					we have
					$$\frac{1}{n!}\int\limits_{\{u<\varphi_j+v\}}(\varphi_j+v-u)^nT
					\leq \int\limits_{\{u<\varphi_j+v\}}(-w_1)d\mu.$$
					Letting $j\rightarrow\infty$ and using the fact $(\lim\limits_{j\rightarrow\infty}\varphi_j)^*=0$, we obtain
					$$\frac{1}{n!}\int\limits_{\{u<v\}}(v-u)^nT\leq \int\limits_{\{u<v\}}(-w_1)d\mu.$$
					The proof is completed.
			\end{proof}
  Similar to Corollary \ref{cor: comparison 1}, we have the following result:
			\begin{corollary}\label{cor: comparison wrt NPMA}
				Let $H\in\pshn$ and $u,v\in\NNP(H)$. Assume that $\NP{u}\geq\NP{v}$. Then $u\leq v.$
			\end{corollary}
	\section{Proofs of Theorem \ref{main} and Corollary \ref{cormain}}\label{secproof}
\subsection{Proof of Theorem \ref{main}} For the reader's convenience, we recall the statement of Theorem \ref{main}.
\begin{theorem}\label{MA w prescribed singularity}
	Assume that there exists $v\in\pshn$ such that $\NP{v}\geq\mu$ and $P[v]=\phi$. Denote 
	$$S=\{w\in\pshn: w\leq  \phi, \NP{w}\geq\mu\}.$$
	Then $u_S:=(\sup\{w: w\in S\})^*$ is a solution of the problem 
		\begin{equation}\label{NPMA1}
		\begin{cases}
		\NP{u}=\mu,\\
		P[u]=\phi,
		\end{cases}
		\end{equation}
	 Moreover, if there exists $\psi\in\NNP$ such that 
	$\NP{\psi}\geq\mu$ then $u_S$ is the unique solution of \eqref{NPMA1}.
\end{theorem}

\begin{proof}  By the assumption, we have $v\leq u_S\leq \phi$ and $P[v]=\phi$. Therefore, $P[u_S]=\phi$.
	We need to show that 	$\NP{u_S}=\mu$.
	
	For every $j \geq 1,$ we denote 
\[
\Omega_j = \{ z \in \Om \colon d(z,\partial \Om) > 2^{-j}\},
\]
and
\[U_j = \{ z \in \Om_j \colon v + \phi > -2^j\}.
\]
We also define
\[
S_{j,k} =\{ w \in \pshn \colon w \leq \phi \text{ on } \Omega \setminus U_k, \NP{w} \geq \mathbb{1}_{U_j} \mu\} ,
\]
for all $k, j \geq 1$.
 It is easy to see that $v \in S_{j,k}$, hence $u_{j,k} := (\sup \{ w\in S_{j,k} \})^*$ is well-defined. Since $S\subset S_{j,k}$, we also
 have
 \begin{equation}\label{eq0.1main}
 u_S \leq u_{j,k}
 \end{equation}
Recall that 
\[
P[\phi] = \Big( \sup \{w \in \pshn \colon w  \leq \phi + O(1) \text{ on } \Om, \liminf\limits_{\Om \setminus \{ \phi = -\infty\} \ni z \to \xi} (\phi(z) - w(z)) \geq 0 \forall \xi \in \partial \Om\} \Big)^*.
\]
By the definition of $S_{j,k}$, we have $u_{j,k} \leq \phi$ on $\Om \setminus \overline{U_k}$ and  $ \phi \geq v + \phi \geq -2^k$
on $\overline{U_k}$. Hence,  $\phi + O(1) \geq u_{j,k}$ on $\Om$ and $\liminf\limits_{\Om \setminus  \{ \phi = -\infty\}\ni z \to \xi} (\phi(z) - u_{j,k}(z)) \geq 0$ for all $ \xi \in \partial \Om$.
 Consequently, we have, $u_{j,k} \leq P[\phi]$. Since $\phi$ is model, it follows that
\begin{equation}\label{eq: u_{j,k}}
 \quad u_{j,k} \leq \phi, \forall k, j \geq 1.
\end{equation}
Moreover, it follows from Theorem \ref{the NPMA of envelope} that
\begin{equation}\label{eq1main}
 \mathbb{1}_{U_k} \NP{u_{j,k}} = \mathbb{1}_{U_k}(\mathbb{1}_{U_j} \mu) = \mathbb{1}_{U_j} \mu,
\end{equation}
for every $k\geq j\geq 1$.

Note that if $j_1\leq j_2$ and $k_1\geq k_2$ then $S_{j_1, k_1}\leq S_{j_2, k_2}$. Hence
\begin{equation}\label{eq2main}
u_{j_1, k_1}\leq u_{j_2, k_2}, \, \forall j_1\leq j_2, k_1\geq k_2.
\end{equation}
Put
\[
u_j = (\lim\limits_{k \to \infty} u_{j,k})^*.
\]
It follows from \eqref{eq0.1main} and \eqref{eq: u_{j,k}} that 
\begin{equation}\label{eq3main}
u_S \leq u_{j} \leq \phi.
\end{equation}
In particular, $u_j\neq -\infty$.
By using \eqref{eq1main} and applying Lemmas \ref{lem: NP of limsup} and 
\ref{thm NPddcu leq u>-infty mu}, we get
\begin{equation}\label{eq: NPu_j = 1_U_j mu}
\quad \mathbb{1}_{U_j} \NP{u_j} = \mathbb{1}_{U_j}\mu, \forall j \geq 1.
\end{equation}

 It follows from \eqref{eq2main} that $(u_j)_{j \geq 1}$ is a decreasing sequence.
Set 
\[
\overline{u} = \lim\limits_{j \to \infty} u_{j}.
\]
By \eqref{eq3main}, we have
\begin{equation}\label{eq4main}
u_S \leq \bar{u} \leq \phi.
\end{equation}
 By using \eqref{eq: NPu_j = 1_U_j mu} and applying Lemmas \ref{lem: NP of limsup} and \ref{thm NPddcu leq u>-infty mu}, we deduce that 
 $$\mathbb{1}_{U_{j_0}} \NP{\overline{u}} = \mathbb{1}_{U_{j_0}} \mu,$$
 for every $j_0\geq 0$.
  Letting $j_0 \to \infty$, we obtain
\begin{equation}\label{eq: NP u ngang}
\mathbb{1}_{\bigcup\limits_{j \geq 1} U_j} \NP{\overline{u}} =\mathbb{1}_{ \bigcup\limits_{j \geq 1} U_j} \mu.
    \end{equation}
By definition, $\Om \setminus \bigcup\limits_{j \geq 1} U_j = \{ v + \phi = -\infty\}$ is a pluripolar set. Therefore, 
 \eqref{eq: NP u ngang} implies that 
\[
\NP{\overline{u}}= \mu.
\]
This combined with  \eqref{eq4main} gives
$$u_S \leq \bar{u} \leq (\sup\{w\in\pshn: w\leq  \phi, \NP{w}\geq\mu\})^*=u_S.$$
Hence, $u_S=\bar{u}$ and $\NP{u_S}= \mu$.
Thus, $u_S$ is a solution of \eqref{NPMA1}.

Now, assume that there exists $\psi \in \mathcal{N}_{NP}$ with $\NP{\psi} \geq \mu.$ We need to show that $u_S$ is the unique solution
of the problem \eqref{NPMA1}. Note that 
$v: = \psi + \phi$ satisfies the conditions $\NP{v} \geq \NP{\psi} \geq \mu$ and $P[v] = \phi.$ Hence  $u_S$ is a solution \eqref{NPMA1} satisfying $$\phi+\psi \leq u_S \leq \phi.$$ 
In particular $u_S \in \mathcal{N}_{NP}(\phi)$.

Let $\mathfrak{u}$ be an arbitrary solution of \eqref{NPMA1}. We will show that   $\mathfrak{u} \in \mathcal{N}_{NP}(\phi)$.

Denote
\[
V_j = \{ z \in \Om_j, \mathfrak{u} > -2^j \},
\]
and 
\[
\mathfrak{u}_j = \Big( \sup \{w \in \pshn \colon w \leq \mathfrak{u} \text{ on } \Om \setminus V_j \} \Big)^*.
\]
 By the same argument as in the proof of Proposition \ref{propmodel}, we have 
\begin{equation}\label{eq5main}
\left(\lim\limits_{j\to\infty}\mathfrak{u}_j\right)^*=P[u]=\phi.
\end{equation}
It is easy to see that
\begin{equation}\label{eq6main}
\mathfrak{u} \leq \mathfrak{u}_j,
\end{equation}
on $\Om$. Moreover, $\mathfrak{u}_j+\psi$ satisfying the conditions
\begin{itemize}
	\item $\mathfrak{u}_j+\psi\leq \mathfrak{u}_j=\mathfrak{u}\quad\mbox{on}\quad \Om\setminus\overline{V_j}$;
	\item $\mathfrak{u}_j+\psi\leq \mathfrak{u}_j\leq \mathfrak{u}+2^j\quad\mbox{on}\quad \overline{V_j}$;
	\item $\NP{(\mathfrak{u}_j+\psi)}\geq\NP{\psi}\geq\mu.$
\end{itemize}
 Then, it follows from Corollary \ref{cor: comparison 1} that
 \begin{equation}\label{eq7main}
 \mathfrak{u}_j+\psi\leq \mathfrak{u}.
 \end{equation}
 Combining \eqref{eq5main}, \eqref{eq6main} and \eqref{eq7main}, we get
 $$\phi+\psi=\left(\lim\limits_{j\to\infty}(\mathfrak{u}_j+\psi)\right)^*\leq \mathfrak{u}\leq 
 \left(\lim\limits_{j\to\infty}\mathfrak{u}_j\right)^*=\phi.$$
 In particular, $\mathfrak{u}\in\mathcal{N}_{NP}(\phi)$. By Corollary \ref{cor: comparison wrt NPMA}, we have $\mathfrak{u}=u_S$.
 Thus, $u_S$ is the unique solution of \eqref{NPMA1}.
 
 This finishes the proof.
\end{proof}
\subsection{Proof of Corollary \ref{cormain}}
In order to prove Corollary \ref{cormain}, we need the following lemma:
\begin{lemma}\label{lemDa}
	Let $u, v, h\in\mathcal{D}(\Om)$ such that $u+v\leq h$. Assume that $(dd^c u)^n$ and $(dd^c v)^n$ vanish on pluripolar sets. 
	Then $(dd^c h)^n$ vanishes on pluripolar set.
\end{lemma}
\begin{proof}
Since the problem is local, we can assume that $\Om$ is hyperconvex and $u, v, h$ are negative. 
In particular, $u, v, h\in\mathcal{E}(\Om)$ (see \cite[Theorem 4.5]{Ceg04} and \cite[Theorem 1.2]{Blo06}). 
Replacing $\Omega$ by a relative compact subset of $\Om$, 
we can also assume that $\int_{\Om} (dd^c w)^n<\infty$ for $w=u, v, h$. 

Let $A\subset\Om$ be a pluripolar set. By \cite[Lemma 4.4]{ACCP} and by the assumption $\int_A(dd^cu)^n=\int_A(dd^cv)^n=0$, we have
$$\int_A (dd^cu)^k\wedge(dd^c v)^{n-k}\leq \left(\int_A (dd^cu)^n\right)^{k/n}\wedge \left(\int_A (dd^cv)^n\right)^{(n-k)/n}=0,$$
for every $k=0, 1, ..., n$. Therefore,
$$\int_A(dd^c(u+v))^n=\sum_{k=0}^{n}\begin{pmatrix}
n\\
k
\end{pmatrix}\int_A(dd^cu)^k\wedge(dd^c v)^{n-k}=0.$$
Since $A$ is arbitrary, we have $(dd^c(u+v))^n$ vanishes on every pluripolar set. Thus, it follows from \cite[Lemma 4.1]{ACCP}
that $(dd^ch)^n$ vanishes on pluripolar sets.
\end{proof}
Now we begin to prove Corollary \ref{cormain}. We recall its statement for the reader's convenience.
	\begin{corollary}
		Assume that $\mu$ is a non-negative measure defined on $\Om$ by $\mu=(dd^c\varphi)^n$ for some $\varphi\in\mathcal{N}^a(\Om)$.
		Then, for every $H\in\mathcal{D}(\Om)$ with $(dd^c H)^n\leq \mu$, there exists a unique function $u\in\mathcal{N}^a(H)$ such that
		$(dd^cu)^n=\mu$ on $\Om$.
	\end{corollary}
\begin{proof}
	Put $\phi=P[H]$ and $u=\left(\sup\{w: w\in S\}\right)^*$, where
	$$S=\{w\in\pshn: P[w]=\phi, (dd^cw)^n\geq\mu\}.$$
	Since $\phi+\varphi\in S$, we have
	\begin{equation}\label{eq1cormain}
	u\geq \phi+\varphi\geq H+\varphi.
	\end{equation}
		By the definitions of $\mathcal{N}^a$ and $\NNP$, we have $\mathcal{N}^a\subset\NNP$. In particular, $\varphi\in\NNP$.
	Then, it follows from Theorem \ref{main} that $u$ is the unique solution to the problem
	\begin{equation}\label{eq1.1cormain}
	\begin{cases}
	\NP{w}=\mu,\\
	P[w]=\phi.
	\end{cases}
	\end{equation}
	Moreover, it follows from \cite[Theorem 1.2]{Blo06} and Lemma \ref{lemDa} that $u\in\mathcal{D}(\Om)$ and $(dd^cu)^n$ vanishes on
	pluripolar sets. Hence, we have
	\begin{equation}\label{eq2cormain}
	(dd^cu)^n=\mu.
	\end{equation}
	Denote $\nu=(dd^cH)^n$. Then $H$ is a solution of the problem
	\begin{equation}\label{eq3cormain}
	\begin{cases}
	\NP{w}=\nu,\\
	P[w]=\phi.
	\end{cases}
	\end{equation}
	Moreover, by Theorem \ref{main}, the problem \eqref{eq3cormain} has a unique solution. Hence
	\begin{equation}\label{eq4cormain}
	H=\left(\sup\{w\in\pshn: P[w]=\phi, (dd^cw)^n\geq\nu \}\right)^*\geq u.
	\end{equation}
	Combining \eqref{eq1cormain} and \eqref{eq4cormain}, we get $u\in\mathcal{N}^a(H)$. This combined with \eqref{eq2cormain} gives
	that $u$ is a solution of the problem
	\begin{equation}\label{eq5cormain}
	\begin{cases}
	w\in\mathcal{N}^a(H),\\
	(dd^cw)^n=\mu.
	\end{cases}
	\end{equation}
	It remains to show the uniqueness of solution of the problem \eqref{eq5cormain}. Assume that $v$ is a solution of
	 \eqref{eq5cormain}. Then there exists $\psi\in\mathcal{N}^a$ such that 
	 $$H+\psi\leq v\leq H.$$
	 Since $\mathcal{N}^a\subset\NNP$, it follows that
	 $$P[H]=P[H]+P[\psi]\leq P[H+\psi]\leq P[v]\leq P[H].$$
	 Then $P[v]=P[H]=\phi$. Moreover, since $\mu=(dd^c\varphi)^n$ vanishes on pluripolar sets,
	  the condition $(dd^cv)^n=\mu$ implies that $\NP{v}=\mu$. Hence, $v$ is a solution of the problem \eqref{eq1.1cormain}.
	  By the uniqueness of solution of \eqref{eq1.1cormain}, we have $v=u$. Thus, $u$ is the unique solution of \eqref{eq5cormain}.
	  
	  The proof is completed.
\end{proof}

	\noindent
	\textbf{Acknowledgments.}
Thai Duong Do was supported by the MOE grant (Singapore) under grant number MOE-T2EP20120-0010. Hoang-Son Do was supported in part by the Vietnam National Foundation for Science and Technology Development (NAFOSTED) under grant number 101.02-2021.16.\\

\noindent
\textbf{Conflict of interest.} The authors declare that there is no conflict of interest.\\
	
\end{document}